\documentclass[12pt,reqno]{amsart}
\usepackage{amsmath,amssymb,amsthm,pinlabel}
\usepackage{tikz}
\usepackage[all,cmtip]{xy}
\usepackage[final, % change to final
colorlinks=true,
pdftitle={\ell^{2}--torsion of free-by-cyclic groups},
pdfauthor={Matt Clay}]{hyperref}
\usepackage{ifthen}

\theoremstyle{plain}
\newtheorem{theorem}{Theorem}
\newtheorem{lemma}[theorem]{Lemma}
\newtheorem{proposition}[theorem]{Proposition}
\newtheorem{corollary}[theorem]{Corollary}

\newtheorem*{claim*}{Claim}

\theoremstyle{definition}
\newtheorem{definition}[theorem]{Definition}

\newtheorem{example}[theorem]{Example}
\newtheorem{remark}[theorem]{Remark}

\newtheorem{convention}[theorem]{Convention}

\numberwithin{theorem}{section}
\numberwithin{equation}{section}

\newcommand{\fakeenv}{} %%% prints the emptystring

\newenvironment{restate}[2]  %%% restate takes two arguments 
{ 
 \renewcommand{\fakeenv}{#2} %%% So now \fakeenv prints #2
 \theoremstyle{plain} 
 \newtheorem*{\fakeenv}{#1~\ref{#2}} %%% so now #2 is the name of a
 \begin{\fakeenv}
}
{
 \end{\fakeenv}
}

\newcommand{\CC}{\mathbb{C}} 
\newcommand{\FF}{\mathbb{F}}
\newcommand{\NN}{\mathbb{N}}

\newcommand{\ZZ}{\mathbb{Z}} 

\newcommand{\calA}{\mathcal{A}}

\newcommand{\sfE}{\mathsf{E}}
\newcommand{\sfV}{\mathsf{V}}

\newcommand{\bxi}{{\boldsymbol \xi}}

\newcommand{\abs}[1]{\left\lvert {#1} \right\rvert} 
\newcommand{\I}[1]{\langle #1 \rangle}
\newcommand{\bd}{\partial}
\newcommand{\from}{\colon\thinspace}

\newcommand{\inv}{^{-1}}
\newcommand{\opnorm}[2][]{\left\| {#2} \right\|_{#1}}

\newcommand{\detG}[1][1]{%
\ifthenelse{\equal{#1}{1}}%
{\operatorname{det}_{G}}%
{\operatorname{det}_{#1}}%
}

\newcommand{\param}%
	{{\mathchoice{\mkern1mu\mbox{\raise2.2pt\hbox{$\centerdot$}}\mkern1mu}%
	{\mkern1mu\mbox{\raise2.2pt\hbox{$\centerdot$}}\mkern1mu}%
	{\mkern1.5mu\centerdot\mkern1.5mu}{\mkern1.5mu\centerdot\mkern1.5mu}}}

\newcommand{\PF}{\mathcal{EG}}

\DeclareMathOperator{\Aut}{Aut}

\DeclareMathOperator{\End}{End}

\DeclareMathOperator{\Mat}{Mat}
\DeclareMathOperator{\Out}{Out}

\DeclareMathOperator{\clos}{clos}

\DeclareMathOperator{\img}{im}
\DeclareMathOperator{\rank}{rk}

\DeclareMathOperator{\Tr}{tr}
\DeclareMathOperator{\vol}{vol}

\begin{document}

%%%%%%%%%%%%%%%%%%%%%%%%%%%%%%%%%%%%%%%%%%%%%%%%%%%%%%%%%%%%%%%%%%%%%%%%%%%%% 

\title{$\ell^{2}$--torsion of free-by-cyclic groups}

\author[M.~Clay]{Matt Clay}
\address{Dept.\ of Mathematics \\
  University of Arkansas\\
  Fayetteville, AR 72701}
\email{\href{mailto:mattclay@uark.edu}{mattclay@uark.edu}}

\thanks{\tiny The author is partially supported by the Simons Foundation.}

\begin{abstract}
We provide an upper bound on the $\ell^{2}$--torsion of a free-by-cyclic group, $-\rho^{(2)}(\FF \rtimes_{\Phi} \ZZ)$, in terms of a relative train-track representative for $\Phi \in \Aut(\FF)$.  Our result shares features with a theorem of L\"uck--Schick computing the $\ell^{2}$--torsion of the fundamental group of a 3--manifold that fibers over the circle in that it shows that the $\ell^{2}$--torsion is determined by the exponential dynamics of the monodromy.  In light of the result of L\"uck--Schick, a special case of our bound is analogous to the bound on the volume of a 3--manifold that fibers over the circle with pseudo-Anosov monodromy by the normalized entropy recently demonstrated by Kojima--McShane.
\end{abstract}

\maketitle

%%%%%%%%%%%%%%%%%%%%%%%%%%%%%%%%%%%%%%%%%%%%%%%%%%%%%%%%%%%%%%%%%%%%%%%%%%%%% 

\section{Introduction}
\label{sec:introduction}

A group $G$ is a \emph{free-by-cyclic group} if it fits into a short exact sequence of the form: \[ 1 \to \FF \to G \to \ZZ \to 1 \] where $\FF$ is a finitely generated free group.
Such a group is a semi-direct product and admits a presentation of the form
\begin{equation}
\label{eq:free-by-cyclic} \FF \rtimes_{\Phi} \ZZ = \I{\FF, t \mid t\inv x t = \Phi(x) \mbox{ for } x \in \FF}
\end{equation}
for some automorphism $\Phi \in \Aut(\FF)$.  Changing $\Phi$ within its outer automorphism  class amounts to replacing the generator $t$ by $tx$ for some $x \in \FF$ and so we are justified in denoting the group in \eqref{eq:free-by-cyclic} by $G_{\phi}$ where $\phi = [\Phi] \in \Out(\FF)$.

These groups share a deep connection with 3--manifolds that fiber over $S^{1}$.  Indeed, assuming for simplicity that the fiber is connected, such a manifold is the \emph{mapping torus}: \[ M_{f} = \raisebox{5pt}{$\Sigma \times [0,1]$} \Big/ \raisebox{-5pt}{$(x,0) \sim (f(x),1)$} \] for a homeomorphism of a connected surface $f \from \Sigma \to \Sigma$ and $\pi_{1}(M_{f}) \cong \pi_{1}(\Sigma) \rtimes_{\Phi} \ZZ$ where $\Phi \in \Aut(\pi_{1}(\Sigma))$ represents the outer automorphism induced by $f$.

As for 3--manifolds that fiber over $S^{1}$, a free-by-cyclic group can be expressed as a semi-direct product in infinitely many ways when the rank of the abelianization is at least two~\cite{ar:Neumann76,ar:Button07}.  The connections between the first homology of a fibered 3--manifold, the topology of the fiber and the dynamics of the monodromy has been extensively studied by  Thurston~\cite{ar:Thurston86}, Fried~\cite{ar:Fried82-1,ar:Fried82-2} and McMullen~\cite{ar:McMullen00}.  The analogous study for free-by-cyclic groups has recently been initiated and developed by Dowdall--Kapovich--Leininger~\cite{ar:DKL15,un:DKL-2} and Algom-Kfir--Hironaka--Rafi~\cite{ar:A-KHR15}.

Thurston's hyperbolization theorem implies that a compact orientable 3--manifold that fibers over $S^{1}$ can be canonically decomposed along incompressible tori such that the components are geometric~\cite{ar:Thurston82}.  We are most interested in those components whose interiors admit a complete hyperbolic metric.  Mostow rigidity implies that the metric in this case is unique~\cite{bk:Mostow73}.  

In the most interesting case when $M_{f}$ does not contain an incompressible torus, the interior of $M_{f}$ admits a unique hyperbolic metric.  For such mapping tori, Brock related the volume, $\vol(M_{f})$, to the translation length of $f$ on the Teichm\"uller space for $\Sigma$ with the Weil--Petersson metric~\cite{ar:Brock03}.  This in particular implies that the volume is bounded above by $\log \lambda(f)$ times a constant that only depends on $\Sigma$.  Here $\lambda(f)$ is the \emph{dilatation} or \emph{stretch factor} of $f$: 
\begin{equation}
\label{eq:dilatation}
\lambda(f) = \sup_{\gamma \subset \Sigma} \lim_{k \to \infty} \root k\of{\ell_{\sigma}(f^{k}(\gamma))} 
\end{equation}
where the supremum is over the set of simple closed curves on $\Sigma$ and $\ell_{\sigma}(\param)$ denotes the length of the unique geodesic in the homotopy class using the hyperbolic metric $\sigma$ on $\Sigma$.  The choice of metric does not matter.  If $f$ is a pseudo-Anosov homeomorphism, the logarithm of $\lambda(f)$ is the topological entropy of $f$~\cite[Expos\'e~10]{bk:FLP79}.

An explicit upper bound for the previously mentioned constant was recently found by Kojima--McShane~\cite{un:KM14}.  If $f \from \Sigma \to \Sigma$ is a pseudo-Anosov homeomorphism they showed that:
\begin{equation}
\label{eq:KM14}
\vol(M_{f}) \leq 3\pi\abs{\chi(\Sigma)}\log \lambda (f).
\end{equation}
A corollary of the main result of this paper is an analog of this equation for free-by-cyclic groups (Corollary~\ref{co:iwip}).  We explain what our replacement for the left-hand side of this inequality is now.

A free-by-cyclic group does not posses a canonical geometry.  Some free-by-cyclic groups act properly discontinuously and cocompactly on CAT(0) or CAT($-1$) spaces and some do not.  By work of Brinkmann~\cite{ar:Brinkmann00} and Bestvina--Feighn~\cite{ar:BF92}, it is known that a free-by-cyclic group $G_{\phi}$ is a word-hyperbolic group if and only if $\phi \in \Out(\FF)$ is \emph{atoroidal}, that is, $\phi$ does not act on the set of non-trivial conjugacy classes of $\FF$ with a periodic orbit.  Neither of these structures can associate to a free-by-cyclic group a canonical notion of ``volume.''

In this paper, we consider an analytic invariant of certain semi-direct products $G \rtimes_{\Phi} \ZZ$ called the \emph{$\ell^{2}$--torsion}, denoted $\rho^{(2)}(G \rtimes_{\Phi} \ZZ)$ (see Section~\ref{sec:auto torsion} for the definition).  This quantity is related to the Alexander polynomial, Reidemeister torsion and Ray--Singer torsion; these connections and further motivation appear in~\cite[Chapter~3]{bk:Luck02}.  Our interest in $\ell^{2}$--torsion is the following special case of a theorem of L\"uck--Schick~\cite{ar:LS99}, see also~\cite[Theorem~4.3]{bk:Luck02}.  If $f \from \Sigma \to \Sigma$ is a homeomorphism of a compact connected surface and $\Phi \in \Aut(\pi_{1}(\Sigma))$ represents the outer automorphism induced by $f$ then $-\rho^{(2)}(\pi_{1}(\Sigma) \rtimes_{\Phi} \ZZ)$ equals $\frac{1}{6 \pi}$ times the sum of the volumes of the hyperbolic components of $M_{f}$.  In particular, the $\ell^{2}$--torsion vanishes exactly when there are no hyperbolic pieces.  Equivalently, the $\ell^{2}$--torsion vanishes exactly if $\lambda(f) = 1$.  

Motivated by this result, we consider $-\rho^{(2)}(G_{\phi})$ as an appropriate analog of volume for a free-by-cyclic group.

Our main result (Theorem~\ref{th:bound}, reprinted below) provides an upper bound on $-\rho^{(2)}(G_{\phi})$ for a free-by-cyclic group $G_{\phi}$; it is already known that $0 \leq -\rho^{(2)}(G_{\phi})$~\cite[Theorem~7.29]{bk:Luck02}.  Our bound employs the Bestvina--Handel theory of relative train-track maps~\cite{ar:BH92} and L\"uck's combinatorial approach to computing the $\ell^{2}$--torsion of free-by-cyclic groups~\cite[Section~7.4.3]{bk:Luck02}.  

The decomposition of an outer automorphism of a free group is more subtle and intricate than the decomposition of a mapping class of a surface.  It is not the case that there is a decomposition of the free group into compatible free factors on which the the outer automorphism acts irreducibly but rather a filtration of the free group where larger elements are allowed to interact via the outer automorphism with the smaller ones but not conversely.  As such, it is not clear what the exact analog of the theorem of L\"uck--Schick should be in this case.  Nonetheless, our bound does decompose into pieces corresponding to the parts of the filtration exhibiting exponential dynamics and in this sense it shares similarities to their result. 

The tool for best understanding the dynamics of a typical outer automorphism are \emph{relative train-track maps} as introduced by Bestvina--Handel~\cite{ar:BH92}.  This is a homotopy equivalence $f \from \Gamma \to \Gamma$ of a finite connected graph representing a given outer automorphism $\phi \in \Out(\FF)$ that respects a filtration $\emptyset = \Gamma_{0} \subset \Gamma_{1} \subset \cdots \subset \Gamma_{S} = \Gamma$ in the sense that $f(\Gamma_{s}) \subseteq \Gamma_{s}$.  For each $1 \leq s \leq S$ there is an associated matrix $M(f)_{s}$, this is the transition matrix for the induced map $\Gamma_{s}/\Gamma_{s-1} \to \Gamma_{s}/\Gamma_{s-1}$.  We may assume that each transition matrix is either the zero matrix or irreducible; in the latter case it has a Perron--Frobenius eigenvalue which we denote $\lambda(f)_{s}$.  The $s$th stratum is \emph{exponentially growing} if $\lambda(f)_{s} > 1$.

When $f \from \Gamma \to \Gamma$ is \emph{irreducible}, i.e., $\Gamma_{1} = \Gamma$, it is known that $\lambda(f)_{1}$ is the \emph{exponential growth rate} or \emph{stretch factor} of $\phi$ (compare \eqref{eq:dilatation}):
\begin{equation}
\label{eq:stretch}
\lambda(\phi) = \sup_{\gamma \in [\FF]} \lim_{k \to \infty} \root k\of{\| \phi^{k}(\gamma) \|_{\calA}}
\end{equation}
where the supremum is over the set of conjugacy classes of elements of $\FF$ and $\| \param \|_{\calA}$ denotes the length of a cyclically reduced word representing the conjugacy class with respect to the basis $\calA$ (see \cite[Remark~1.8]{ar:BH92}).  The choice of basis does not matter.  If $f$ is an irreducible train-track, the logarithm of $\lambda(f)_{1}$ is the topological entropy of $f$~\cite[Proposition~2.8]{ar:DKL15}.

With this, we can now state our main result.  

\begin{restate}{Theorem}{th:bound}
Suppose $f \from \Gamma \to \Gamma$ is a relative train-track map with respect to the filtration $\emptyset = \Gamma_{0} \subset \Gamma_{1} \subset \cdots \subset \Gamma_{S} = \Gamma$ representing the outer automorphism $\phi \in \Out(\FF)$.  Let $n_{s}$ denote the number of edges of $\Gamma_{s} - \Gamma_{s-1}$, $\lambda(f)_{s}$ the Perron--Frobenius eigenvalue of $M(f)_{s}$ and $\PF(f)$ the set of indices of the exponentially growing stratum. Then:
\begin{equation*}
-\rho^{(2)}(G_{\phi}) \leq \sum_{s \in \PF(f)} n_{s} \log \lambda(f)_{s}.
\end{equation*} 
In particular, if $f \from \Gamma \to \Gamma$ does not have an exponentially growing stratum then $\rho^{(2)}(G_{\phi}) = 0$.
\end{restate}

Since the number of edges in $\Gamma$ is bounded by $3\rank(\FF) - 3 = 3\abs{\chi(\FF)}$, as a corollary we obtain an inequality akin to the inequality of Kojima--McShane \eqref{eq:KM14}.

\begin{corollary}
\label{co:iwip}
Suppose $\phi \in \Out(\FF)$ can be represented by an irreducible train-track map.  Then:
\begin{equation}
\label{eq:iwip}
-\rho^{(2)}(G_{\phi}) \leq 3\abs{\chi(\FF)} \log \lambda(\phi).
\end{equation}
\end{corollary}

Kin--Kojima--Takasawa exhibited a positive lower bound on $\vol(M_{f})$, specifically $\log \lambda(f)$ times a positive constant that depends on $\Sigma$ and the length of a systole of $M_{f}$~\cite{ar:KKT09}.  It would be very interesting to exhibit an analogous lower bound on $-\rho^{(2)}(G_{\phi})$ or even a criterion for positivity.  Conjecturally, $-\rho^{(2)}(G_{\phi}) > 0$ if $\PF(\phi) \neq \emptyset$.  An explicit lower bound could provide information on the ratio $\frac{\log \lambda(\phi)}{\log \lambda(\phi\inv)}$ (cf.~\cite{ar:HM07-2,ar:DKL14}).    

\medskip 

This paper is organized as follows.  Section~\ref{sec:graph} provides background on relative train-track maps.  The necessary $\ell^{2}$--theory needed to define $\ell^{2}$--torsion is briefly introduced in Section~\ref{sec:torsion}.  In Section~\ref{sec:compute} we tailor L\"uck's combinatorial approach to the setting of relative train-track maps.  The proof of Theorem~\ref{th:bound} appears in Section~\ref{sec:bounds}.

%%%%%%%%%%%%%%%%%%%%%%%%%%%%%%

\subsection*{Acknowledgments} The author thanks Andy Raich for discussions related to this work, Spencer Dowdall for explaining his work with Kapovich and Leininger, and the referee for helpful suggestions.

%%%%%%%%%%%%%%%%%%%%%%%%%%%%%%%%%%%%%%%%%%%%%%%%%%%%%%%%%%%%%%%%%%%%%%%%%%%%% 

\section{Topological representatives of outer automorphisms}
\label{sec:graph}

In this section, we introduce notation and collect facts and theorems about topological representatives of outer automorphisms of free groups.

%%%%%%%%%%%%%%%%%%%%%%%%%%%%%%

\subsection{Graphs} By \emph{graph} we mean a 1--dimensional CW--complex.  If $\Gamma$ is a graph, we let $\sfV(\Gamma)$ denote the set of vertices (0-cells) and $\sfE(\Gamma)$ denote the set of edges (1-cells).  As edges are 1--cells, they are oriented; denote the endpoints of $e \in \sfE(\Gamma)$ by $\bd_{0}(e)$ (initial vertex) and $\bd_{1}(e)$ (terminal vertex).  The same edge with reversed orientation is denoted $e\inv$.  

Graphs will always be assumed to be connected and not to have vertices of valence 1.

An \emph{edge-path} is the image of a cellular map $p \from [0,1] \to \Gamma$ for which there exists a partition $0 = x_{0} < \cdots < x_{k} = 1$ such that $p\big|_{[x_{i-1},x_{i}]}$ is an orientation preserving homeomorphism onto an edge $e_{i}^{\varepsilon_{i}}$ where $\varepsilon_{i} \in \{ 1, -1\}$.  As such, we write $p([0,1]) = \prod_{i=1}^{k} e_{i}^{\varepsilon_{i}}$.  An edge-path $p([0,1]) = \prod_{i=1}^{k} e_{i}^{\varepsilon_{i}}$ is \emph{reduced} if for all $1 \leq i < k$, we have $e_{i}^{\varepsilon_{i}} \neq e_{i+1}^{-\varepsilon_{i+1}}$, i.e., $p$ is locally injective.  An edge-path is homotopic rel $\{0,1\}$ to a unique reduced edge-path.  

%%%%%%%%%%%%%%%%%%%%%%%%%%%%%%

\subsection{Morphisms} A \emph{morphism} of graphs $f \from \Gamma \to \Gamma'$ is a cellular map such that $f$ linearly expands (with respect to some metrics) each edge of $\Gamma$ across an edge-path of $\Gamma'$.  A morphism is homotopic rel $\sfV(\Gamma)$ to a unique morphism so that the image of every edge is a reduced edge-path.  We will implicitly make the assumption throughout that morphisms are so reduced.

Let $\Gamma$ be a finite graph and fix an ordering of the edges $\sfE(\Gamma) = \{e_{1},\ldots,e_{n} \}$.  Given a morphism $f \from \Gamma \to \Gamma$, the \emph{transition matrix} $M(f)$ is the $n \times n$ matrix with non-negative integer entries $m_{i, \, j}$ defined as the number of times $e_{j}^{\pm 1}$ appears in the edge-path $f(e_{i})$.  Notice that by our assumption that morphisms are homotoped so that the image of every edge is a reduced edge-path, in general $M(f^{k})$ need not equal $M(f)^{k}$.  There is a setting in which a version of this can hold, which is the content of the next section.

%%%%%%%%%%%%%%%%%%%%%%%%%%%%%%

\subsection{Relative train-track maps}\label{sec:rtt} Let $\Gamma$ be a graph, $v \in \sfV(\Gamma)$ and suppose there is an identification $\pi_{1}(\Gamma,v) \cong \FF$.  Then a homotopy equivalence $f \from \Gamma \to \Gamma$ specifies an outer automorphism of $\FF$.  Conversely, every outer automorphism $\phi \in \Out(\FF)$ can be represented by a morphism that is homotopy equivalence of $\Gamma$.  A \emph{relative train-track} map is a particular topological representative of a given outer automorphism that is useful for studying its dynamical properties.  Such maps were defined and shown to exist by Bestvina--Handel~\cite[Theorem~5.12]{ar:BH92} and have been essential in many subsequent results and discoveries for $\Out(\FF)$.  We will not recall their complete definition here, only the relevant aspects needed for this article.

Suppose $f \from \Gamma \to \Gamma$ is a morphism that respects a filtration by subgraphs $\emptyset = \Gamma_{0} \subset \Gamma_{1} \subset \cdots \subset \Gamma_{S} = \Gamma$ in the sense that $f(\Gamma_{s}) \subseteq \Gamma_{s}$ for all $0 \leq s \leq S$.  The $s$th \emph{stratum} $H_{s}$ is the closure of $\Gamma_{s} - \Gamma_{s-1}$.  We denote the number of edges in $H_{s}$ by $n_{s}$.   We order the set of edges $\sfE(\Gamma) = \{ e_{1},\ldots,e_{n} \}$ such that if $e_{i} \in H_{i'}$, $e_{j} \in H_{j'}$ and $i' < j'$ then $i < j$, i.e., edges on lower strata are lower in the order.  For each $1 \leq s \leq S$, we let $i_{s}$ denote the smallest index such that $e_{i_{s}} \in \sfE(H_{s})$ and $M(f)_{s}$ the $n_{s} \times n_{s}$--submatrix of $M(f)$ with $i_{s} \leq i,j \leq i_{s} - 1 + n_{s}$; thus $M(f)_s$ is the transition matrix of the induced map $\Gamma_s/\Gamma_{s-1} \to \Gamma_s/\Gamma_{s-1}$\footnote{For this to be true, do not homotope the induced map to eliminate nonreduced edge-paths.}.  With these conventions, $M(f)$ is lower block triangular with blocks $M(f)_{s}$ along the diagonal.  

After possibly increasing the size of the filtration, we can assume that each $M(f)_{s}$ is either the zero matrix or is \emph{irreducible}: for each $1 \leq i,j \leq n_{s}$ there is a $k \in \NN$ such that the $ij$th entry of $\bigl(M(f)_{s}\bigr)^{k}$ is not zero.  

The key property of relative train-track maps that we need for the sequel is the following.

\begin{lemma}[{\cite[Lemma~5.8]{ar:BH92}}]
\label{lem:rtt}
Suppose $f \from \Gamma \to \Gamma$ is a relative train-track map with respect to the filtration $\emptyset = \Gamma_{0} \subset \Gamma_{1} \subset \cdots \subset \Gamma_{S} = \Gamma$.  Then for all $1 \leq s \leq S$ and $k \in \NN$, we have $M(f^{k})_{s} = \bigl(M(f)_{s}\bigr)^{k}$.
\end{lemma}

In other words, if $e \in \sfE(H_{s})$, then the edges in $H_{s}$ crossed by $f^{k}(e)$ are not canceled when $f^{k}(e)$ is homotoped rel endpoints to a reduced edge-path.

To each irreducible $M(f)_{s}$ is associated a Perron--Frobenius eigenvalue $\lambda(f)_{s} \geq 1$.  We set $\PF(f) = \{ s \mid M(f)_{s} \mbox{ is irreducible and } \lambda(f)_{s} > 1 \}$; these are the indices of the \emph{exponentially growing stratum}.  The set of eigenvalues $\{\lambda(f)_{s} \mid s \in \PF(f) \}$ only depends on the outer automorphism represented by $f$.  

If $f \from \Gamma \to \Gamma$ is a relative train-track map with a single stratum, the adjective ``relative'' is replaced by ``irreducible.''  This happens in particular if $\phi \in \Out(\FF)$ is \emph{irreducible}.  This means that $\phi$ does not cyclically permute the conjugacy classes of free factors $\FF_{i}$, $1 \leq i \leq k$ where $\FF = \FF_{1} \ast \cdots \ast \FF_{k} \ast \FF'$.  We note that it is possible that a reducible $\phi \in \Out(\FF)$ can be represented by an irreducible train-track map and that powers of irreducible outer automorphisms need not be irreducible.  

%%%%%%%%%%%%%%%%%%%%%%%%%%%%%%%%%%%%%%%%%%%%%%%%%%%%%%%%%%%%%%%%%%%%%%%%%%%%%

\section{\texorpdfstring{$\ell^{2}$}{l\texttwosuperior}--torsion}
\label{sec:torsion}

For this section, let $G$ be a countable group.  We briefly introduce the necessary $\ell^{2}$--theory needed for the remainder of this article.  There are excellent surveys on $\ell^{2}$--homology of discrete groups by Eckmann~\cite{ar:Eckmann00} and L\"uck~\cite{col:Luck02}.  Additionally, L\"uck's book~\cite{bk:Luck02} is a comprehensive reference on the subject and much of the material below is taken from this source.  One may safely skim this section on a first read and refer back when various definitions and theorems are used in later sections.

%%%%%%%%%%%%%%%%%%%%%%%%%%%%%%

\subsection{Hilbert spaces, operators and norms}  The Hilbert space of square-summable functions $\xi \from G \to \CC$ is denoted $\ell^{2}(G)$.  The inner product on $\ell^{2}(G)$ is given by:
\[ \I{\xi_{1},\xi_{2}} = \sum_{g \in G} \xi_{1}(g)\overline{\xi_{2}(g)} \]  
The associated $\ell^{2}$--norm is denoted $\| \xi \| = \I{\xi,\xi}^{1/2}$.  The dense subspace of finitely supported functions is isomorphic (as a vector space) to the group algebra $\CC[G]$.  As such, we consider $g \in G$ as the element of $\CC[G] \subset \ell^{2}(G)$ that is the unit function that takes value 1 on $g$ and 0 elsewhere.  

Recall that if $T \from U \to V$ is an operator between Hilbert spaces, then the \emph{operator norm} is defined as $\opnorm{T} = \sup_{\| x \|_{U} = 1} \| T(x) \|_{V}$.  The operator $T$ is \emph{bounded} if $\opnorm{T} < \infty$.

The group $G$ acts linearly and isometrically on both the left and on the right of $\ell^{2}(G)$ by:
\[ (g \cdot \xi)(h) = \xi(g\inv h) \mbox{ and } (\xi \cdot g)(h) = \xi(hg\inv) \]
By linearity, these extend to left and right actions of $\CC[G]$ by bounded operators.  In fact, if $\xi \in \ell^{2}(G)$ and $a \in \CC[G]$, we have $\| \xi \cdot a \| \leq \| \xi \| \cdot \abs{a}$  where $\abs{\param} \from \CC[G] \to \CC$ is the $\ell^{1}$--norm.  %In other words, $\| R_{a} \| \leq \abs{a}$ where $R_{a} \from \ell^{2}(G) \to \ell^{2}(G)$ is the operator defined by $R_{a}(f) = f \cdot a$ and $\| \param \|$ denotes the operator norm.  
  
The adjoint of the operator associated to $a \in \CC[G]$ is the operator associated to the \emph{conjugate} $\overline{a} \in \CC[G]$, where for $\xi \in \ell^{2}(G)$ we define its conjugate by the formula $\overline{\xi}(g) = \overline{\xi(g\inv)}$. 

The \emph{trace} of an operator $T \from \ell^{2}(G) \to \ell^{2}(G)$ is $\Tr_{G}(T) = \I{T(1),1}$ where $1 \in G$ is the identity element.  This notion extends to a matrix of operators $T = [ T_{i, \, j} ] \from \bigl( \ell^{2}(G) \bigr)^{n} \to \bigl( \ell^{2}(G) \bigr)^{n}$ in the usual way, that is $\Tr_{G}(T) = \sum_{i=1}^{n} \Tr_{G}(T_{i, \, i})$.

For Theorem~\ref{th:bound}, we need an estimate on the operator norm of the operator $T$ defined on $\bigl(\ell^{2}(G)\bigr)^{n}$ by right-multiplication with a matrix $A = [a_{i, \, j}] \in \Mat_{n}(\CC[G])$.  We remark for future reference that the adjoint of such an operator is the operator defined by right-multiplication with $A^{*} = \left[\overline{a_{j, \, i}}\right]$.  
     
%Abusing notation, we will let $A$ denote both the matrix in $\Mat_{n}(\CC[G])$ and the operator on $\bigl(\ell^{2}(G)\bigr)^{n}$ given by right-multiplication.  We remark for future reference that the adjoint of such an operator is right-multiplication by $A^{*} = \left[\overline{a_{j, \, i}}\right]$.  

For $A = [ a_{i, \, j} ] \in \Mat_{n}(\CC[G])$, we define $A_{\ell^{1}} \in \Mat_{n}(\CC)$ by $A_{\ell^{1}} = \left[ | a_{i, \, j} | \right]$. 

\begin{proposition}
\label{prop:operator bound}
Suppose $A \in \Mat_{n}(\CC[G])$.  Let $T \from \bigl(\ell^{2}(G)\bigr)^{n} \to \bigl(\ell^{2}(G)\bigr)^{n}$ be the operator induced by right-multiplication by $A$ and $T_{\ell^{1}} \from \CC^{n} \to \CC^{n}$ the operator induced by right-multiplication by $A_{\ell^{1}}$.  Then $\opnorm{T} \leq \opnorm{T_{\ell^{1}}}$.
\end{proposition}

\begin{proof}
Fix an $n$-tuple $\boldsymbol{\xi} = (\xi_{1},\ldots,\xi_{n}) \in \bigl(\ell^{2}(G)\bigr)^{n}$.  We compute:
\begin{align*}
\| T(\boldsymbol{\xi}) \|^{2} & = \bigl\| (\xi_{1},\ldots,\xi_{n})A \bigr\|^{2}  = \sum_{i=1}^{n} \left\| \sum_{j=1}^{n} \xi_{j} \cdot a_{j, \, i}\right\|^{2} \\
& \leq \sum_{i=1}^{n} \sum_{j=1}^{n} \| \xi_{j} \cdot a_{j, \, i} \|^{2} \leq \sum_{i=1}^{n} \sum_{j=1}^{n} (\| \xi_{j} \| \cdot |a_{j, \, i}|)^{2} \\
&\leq \sum_{i=1}^{n} \left( \sum_{j=1}^{n} \| \xi_{j} \| \cdot |a_{j, \, i}| \right)^{2}  = \bigl\| (\| \xi_{1} \|, \ldots, \| \xi_{n} \|)A_{\ell^{1}} \bigr\|^{2} \\
& \leq \bigl\| (\| \xi_{1} \|, \ldots, \| \xi_{n} \|) \bigr\|^{2} \opnorm{T_{\ell^{1}}}^{2}  =\bigl\| \boldsymbol{\xi} \bigr\|^{2}  \opnorm{T_{\ell^{1}}}^{2} 
\end{align*}
The result follows.
\end{proof}

We recall a standard fact about matrix norms on $\Mat_{n}(\CC)$.  We specifically use the norm $\opnorm{A}$ defined as the operator norm of $T \from \CC^{n} \to \CC^{n}$ induced by right-multiplication by $A$.  For any $A \in \Mat_{n}(\CC)$ we have:
\begin{equation}
\label{eq:matrix norm}
\lim_{k \to \infty} \opnorm{A^{k}}^{1/k} = r(A)
\end{equation}
where $r(A)$ is the spectral radius of $A$, i.e., largest absolute value of an eigenvalue of $A$.  See for instance~\cite[Theorem~VI.6]{bk:RS80}.  This equation easily implies the following limit which is needed for Theorem~\ref{th:bound}.   

\begin{lemma}
\label{lem:I + A}
Suppose that $A \in \Mat_{n}(\CC)$ with $r(A) > 1$ and let $I$ denote the identity matrix in $\Mat_{n}(\CC)$.  Then:
\begin{equation*}
\lim_{k \to \infty} \opnorm{I + A^{k}}^{1/k} = r(A).
\end{equation*}
\end{lemma}

\begin{proof}
By the triangle inequality we have:
\[\opnorm{A^{k}} - 1 \leq \opnorm{I + A^{k}} \leq 1 + \opnorm{ A^{k}}.\]  Dividing both sides by $\opnorm{A^{k}}$ and taking the $k^{\rm th}$ root we find:
\[ \left(1 - \frac{1}{\opnorm{A^{k}}}\right)^{1/k} \leq \left( \frac{\opnorm{I + A^{k}}}{\opnorm{ A^{k}}} \right)^{1/k} \leq \left(1 + \frac{1}{\opnorm{A^{k}}}\right)^{1/k} \]
As $r(A) > 1$, we have $\opnorm{A^{k}} \to \infty$ and hence 
\[ \lim_{k \to \infty} \opnorm{I + A^{k}}^{1/k} = \lim_{k \to \infty} \opnorm{A^{k}}^{1/k} = r(A) \]
using \eqref{eq:matrix norm}.
\end{proof}

%%%%%%%%%%%%%%%%%%%%%%%%%%%%%%

\subsection{\texorpdfstring{$\ell^{2}$}{l\texttwosuperior}--homology}\label{sec:homology}

A \emph{Hilbert--$G$--module}\footnote{This definition is sometimes called a \emph{finitely generated} Hilbert--$G$--module.} is a Hilbert space $V$ equipped with an action of $G$ by linear isometries for which there exists an $G$--equivariant isometric linear embedding $V \to \bigl(\ell^{2}(G)\bigr)^{n}$ for some $n$, with regard to the left $G$--action on $\bigl(\ell^{2}(G)\bigr)^{n}$.  A \emph{map} of Hilbert--$G$--modules $T \from U \to V$ is a bounded $G$--equivariant operator.  Given a chain complex $C_{*}^{(2)} = \{ c_{p} \from C_{p}^{(2)} \to C_{p-1}^{(2)} \}_{p \in \ZZ}$ of Hilbert--$G$--modules, the \emph{$\ell^{2}$--homology} is defined by:
\[ H^{(2)}_{p}(C_{*}^{(2)}) = \raisebox{5pt}{$\ker c_{p}$} \big/ \raisebox{-5pt}{$\clos(\img c_{p+1})$}. \]
Taking the quotient by the closure ensures that the resulting object is a Hilbert--$G$--module.

Suppose $X$ is a CW--complex equipped with an action of $G$ that freely permutes the cells such that there are only finitely many orbits of cells in each dimension.  Let $C_{*}(X)$ be the usual cellular chain complex of $X$.  Hence the chain groups are free $\ZZ[G]$--modules of finite rank.  Then $C_{*}^{(2)}(X) = \ell^{2}(G) \otimes_{\ZZ[G]} C_{*}(X)$ is a chain complex of Hilbert--$G$--modules and we define: 
\[H_{p}^{(2)}(X) = H_{p}^{(2)}(C_{*}^{(2)}(X)).\]
A $G$--equivariant homotopy equivalence $f \from X \to Y$ of CW--complexes equipped with free $G$--actions as above induces an isomorphism $H_{p}^{(2)}(X) \to H_{p}^{(2)}(Y)$~\cite[Theorem~1.35(1)]{bk:Luck02}.  Hence if there exists an classifying space for $G$, $BG$, that has finitely many cells in each dimension we are justified in defining:
\[ H_{p}^{(2)}(G) = H_{p}^{(2)}(EG). \]

We will need the following special case of a theorem of L\"uck.

\begin{theorem}[{\cite[Theorem~1.39]{bk:Luck02}}]\label{th:vanish}
Let $\Phi \from \FF \to \FF$ be an automorphism.  Then for all $p \geq 0$:
\[ H^{(2)}_{p}(\FF \rtimes_{\Phi} \ZZ) = 0.  \]
\end{theorem}

This theorem holds more generally for $G \rtimes_{\Phi} \ZZ$ whenever $G$ has an Eilenberg--Maclane space with finitely many cells in each dimension and in a wider setting as well.

%%%%%%%%%%%%%%%%%%%%%%%%%%%%%%

\subsection{Fuglede--Kadison determinant}\label{sec:det}

Suppose $T \from U \to V$ is map of Hilbert--$G$--modules and let $\nu_{T} \from [0,\infty) \to [0,\infty)$ be the spectral density function of $T^{*}T$.  We will not make use of the definition of $\nu$ but for reference for the reader state that in the case that $G$ is the trivial group and $T$ is induced by right-multiplication with a matrix $A \in \Mat_{n}(\CC)$ the function $\nu$ is the right-continuous step function with jumps at the square root of each eigenvalue of $AA^{*}$ of size the algebraic multiplicity of the eigenvalue.  See \cite[Section~2.1]{bk:Luck02} for complete details.

The \emph{Fuglede--Kadison determinant} of $T$ is defined by:
\begin{equation}
\label{eq:det} 
\detG(T) = {\rm exp} \int_{0^{+}}^{\infty} \log(\lambda) \, d\nu_{T} 
\end{equation}  
if the integral exists, else $\detG(T) = 0$.  If the context is clear, we will omit $G$ from the notation.  

In the case described above when $G$ is the trivial group, the reader can verify using the description of $\nu_{T}$ that $\det(T) = \sqrt{\det(AA^{*})}$ if $\det(AA^{*}) > 0$.

As for the usual determinant of matrices over $\CC$, the operator norm and dimension can be used to bounded the determinant.

\begin{lemma}[{\cite[Lemma~6.9]{un:Luck}}]\label{lem:det-bound}
Suppose $T \from \bigl(\ell^{2}(G) \bigr)^{m} \to \bigl(\ell^{2}(G) \bigr)^{n}$ is a map of Hilbert--$G$--modules.  Then:
\[ \det(T) \leq \opnorm{T}^{n}. \]
\end{lemma}

An important property of the Fuglede--Kadison determinant is that under certain circumstances, it can be computed using block forms.  The special case of this statement we need is the following.  Suppose 
\begin{align*}
T_{1} &\from \bigl(\ell^{2}(G) \bigr)^{m} \to \bigl(\ell^{2}(G) \bigr)^{m}, \\ 
T_{2} &\from \bigl(\ell^{2}(G) \bigr)^{n} \to \bigl(\ell^{2}(G) \bigr)^{n}, \mbox{ and} \\ 
T_{3} &\from \bigl(\ell^{2}(G) \bigr)^{n} \to \bigl(\ell^{2}(G) \bigr)^{m}
\end{align*}
are bounded $G$--equivariant operators where $T_{1}$ and $T_{2}$ are injective.    Then for the operator $T \from \bigl(\ell^{2}(G) \bigr)^{m+n} \to \bigl(\ell^{2}(G) \bigr)^{m+n}$ defined by \[T(\bxi_{1},\bxi_{2}) = \bigl((T_{1}(\bxi_{1}) + T_{3}(\bxi_{2}),T_{2}(\bxi_{2})\bigr),\] we have 
\begin{equation}
\label{eq:det-block}
\det T = \det (T_{1}) \cdot \det(T_{2})
\end{equation}
See \cite[Theorem~3.14(2)]{bk:Luck02}.

%%%%%%%%%%%%%%%%%%%%%%%%%%%%%%

\subsection{\texorpdfstring{$\ell^{2}$}{l\texttwosuperior}--torsion of group automorphisms}\label{sec:auto torsion}

The $\ell^{2}$--torsion is an invariant of a chain complex of Hilbert--$G$--modules $C_{*}^{(2)} = \{ c_{p} \from C_{p+1}^{(2)} \to C_{p}^{(2)}\}_{p \in \ZZ}$.  It is defined as the alternating sum of determinants of the operators $c_{p}$.  

\begin{definition}
\label{def:torsion}
Let $C_{*}^{(2)} = \{ c_{p} \from C_{p}^{(2)} \to C_{p-1}^{(2)} \}_{p \in \ZZ}$ be a chain complex of Hilbert--$G$--modules such that $C_{p}^{(2)}$ is nontrivial for only finitely many $p$ and $\det (c_{p}) \neq 0$ for all $p$\footnote{Note that by definition, the determinant of the zero map is 1.}.  The \emph{$\ell^{2}$--torsion of $C_{*}^{(2)}$} is defined by:
\begin{equation}
\label{eq:torsion}
\rho^{(2)}(C_{*}^{(2)}) = -\sum_{p \in \ZZ} (-1)^{p} \log \det (c_{p}).
\end{equation}
\end{definition}

We need the following special case of a sum formula for the $\ell^{2}$--torsion of chain complexes.

\begin{theorem}[{\cite[Theorem~3.35(1)]{bk:Luck02}}]
\label{th:torsion sum}
Suppose that \[0 \to B_{*}^{(2)} \to C_{*}^{(2)} \to D_{*}^{(2)} \to 0\] is an exact sequence of chain complexes of Hilbert--$G$--modules satisfying the assumptions of Definition~\ref{def:torsion}.  Further suppose that for each $p$:
\begin{enumerate}
\item $H_{p}^{(2)}(B_{*}) = H_{p}^{(2)}(C_{*}) = H_{p}^{(2)}(D_{*}) = 0$ and
\item $C_{p}^{(2)} = B_{p}^{(2)} \oplus D_{p}^{(2)}$.
\end{enumerate}
Then $\rho^{(2)}(C_{*}^{(2)}) = \rho^{(2)}(B_{*}^{(2)}) + \rho^{(2)}(D_{*}^{(2)})$.
\end{theorem}

%\begin{theorem}[{\cite[Theorem~3.35(1)]{bk:Luck02}}]
%\label{th:torsion sum}
%Suppose $B_{*}^{(2)}$ and $D_{*}^{(2)}$ are weakly acyclic chain complexes of Hilbert--$G$--modules (i.e., their $\ell^{2}$--homology vanishes) satisfying the assumptions of Definition~\ref{def:torsion} and consider their sum $C_{*}^{(2)} = B_{*}^{(2)} \oplus D_{*}^{(2)}$.  Then $\rho^{(2)}(C_{*}^{(2)}) = \rho^{(2)}(B_{*}^{(2)}) + \rho^{(2)}(D_{*}^{(2)})$.
%\end{theorem}

Suppose that $G$ is residually finite and has a finite classifying space, $BG$. Fix an automorphism $\Phi \in \Aut(G)$.  In this case, the $\ell^{2}$--torsion of the chain complex $C_{*}^{(2)}(E(G \rtimes_{\Phi} \ZZ))$ is well-defined and only depends on the group $G \rtimes_{\Phi} \ZZ$~\cite[Lemma~13.6]{bk:Luck02} and so we are justified in defining:
\[ \rho^{(2)}(G \rtimes_{\Phi} \ZZ) = \rho^{(2)}(C_{*}^{(2)}(E(G \rtimes_{\Phi} \ZZ))). \]  This invariant behaves in certain respects like the Euler characteristic.  Theorem~\ref{th:torsion sum} is one such example of this. Of importance in the present situation is the following theorem.

\begin{theorem}[{\cite[Theorem~7.27 (4)]{bk:Luck02}}]
\label{th:torsion}
Let $\Phi \from \FF \to \FF$ be an automorphism.  Then for all $k \in \NN$:
\[ \rho^{(2)}(\FF \rtimes_{\Phi^{k}} \ZZ) = k\rho^{(2)}(\FF \rtimes_{\Phi} \ZZ).  \]
\end{theorem}

More generally the theorem holds for $G \rtimes _{\Phi} \ZZ$ when $G$ is residually finite and has a finite classifying space and moreover in a wider setting as well, see~\cite[Section~7.4.1]{bk:Luck02}.  

As mentioned in the introduction, L\"uck--Schick proved that the $\ell^{2}$--torsion $-\rho^{(2)}(\pi_{1}(\Sigma) \rtimes_{\Phi} \ZZ)$ equals $\frac{1}{6\pi}$ times the sum of the volumes of the hyperbolic components of $M_{f}$ where $\Phi \in \Aut(\pi_{1}(\Sigma))$ represents that outer automorphism induced by $f \from \Sigma \to \Sigma$~\cite{ar:LS99}.  Hence we view $-\rho^{(2)}(\FF \rtimes_{\Phi} \ZZ)$ as the appropriate analog of the hyperbolic volume of a free-by-cyclic group.

%%%%%%%%%%%%%%%%%%%%%%%%%%%%%%%%%%%%%%%%%%%%%%%%%%%%%%%%%%%%%%%%%%%%%%%%%%%%%

\section{Computing torsion from a topological representative}
\label{sec:compute}

The goal of this section is two-fold.  First, we will prove Theorem~\ref{th:det-splitting} that shows that the $\ell^{2}$--torsion $-\rho^{(2)}(G_{\phi})$ can be computed using information encoded in the strata $H_{s}$ of a relative train-track map.  (Recall $G_{\phi} = \FF \rtimes_{\Phi} \ZZ$ where $\phi$ is the outer automorphism class of $\Phi$.) Secondly, we will prove Corollary~\ref{co:compute} that gives an upper bound on $-\rho^{(2)}(G_{\phi})$ in terms of the norm of certain operators on the Hilbert spaces $\bigl(\ell^{2}(G_{\phi}) \bigr)^{n}$ related to the strata $H_{s}$.  This bound is improved in the next section.

%%%%%%%%%%%%%%%%%%%%%%%%%%%%%%

\subsection{The Jacobians \texorpdfstring{\boldmath $J_{0}(f)$}{J\textunderscore 0(f)} and \texorpdfstring{\boldmath $J_{1}(f)$}{J\textunderscore 1(f)}}\label{sec:jacobian} Let $\Gamma$ be a graph and $f \from \Gamma \to \Gamma$ a morphism.  Fix a vertex $v \in \sfV(\Gamma)$, let $n = \#\abs{\sfE(\Gamma)}$, and let $\tau \subseteq \Gamma$ be a maximal subtree.    The quotient map $\Gamma \to \Gamma/\tau \cong \bigvee_{i=1}^{1 - \chi(\Gamma)} S^{1}$ induces an isomorphism $\pi_{1}(\Gamma,v) \cong \FF$, where we have a fixed identification between $\pi_{1}\left( \bigvee_{i=1}^{1 - \chi(\Gamma)} S^{1},x_{0} \right)$ and $\FF$.  A choice of an edge-path $p$ from $v$ to $f(v)$ induces an endomorphism of $\FF$ by $\gamma \mapsto p \cdot f(\gamma) \cdot p\inv$.  Denote this endomorphism by $\Phi$.    

Fix an arbitrary order on $\sfV(\Gamma) = \{v_{1},\ldots,v_{m}\}$ and $\sfE(\Gamma) = \{e_{1},\ldots,e_{n}\}$.  For each vertex $v_{i} \in \sfV(\Gamma)$, let $\alpha_{i} \subseteq \tau$ be the (possibly trivial) edge-path from $v$ to $v_{i}$.  Similarly, for each edge $e_{i} \in \sfE(\Gamma)$, let $\beta_{i} \subseteq \tau$ be the (possibly trivial) edge-path from $v$ to $\partial_{0} e_{i}$.    

Fix a lift of $v$ to $\tilde v \in \widetilde{\Gamma}$.  Let $\tilde p$ be the lift of $p$ to $\widetilde{\Gamma}$ starting at $\tilde v$ and let $\tilde w = \bd_{1}(\tilde p)$.  Consider the lift $f \from \Gamma \to \Gamma$ to $\tilde f \from \widetilde{\Gamma} \to \widetilde{\Gamma}$ such that $\tilde f(\tilde v) = \tilde w$.  Using this lift, we have $\tilde f (g x) = \Phi(g) \tilde f(x)$ for all $g \in \FF$ and $x \in \widetilde{\Gamma}$.  

For each $v_{i} \in \sfV(\Gamma)$, we let $\tilde \alpha_{i}\subset \widetilde{\Gamma}$ be the lift of $\alpha_{i}$ starting at $\tilde v$.  Similarly define lifts $\tilde \beta_{i} \subset \widetilde{\Gamma}$.  Let $\tilde v_{i} = \partial_{1}(\tilde \alpha_{i})$ and let $\tilde e_{i} \subset \widetilde{\Gamma}$ be the lift of $e_{i}$ with $\partial_{0}(\tilde e_{i}) = \partial_{1}(\tilde{\beta}_{i})$.  This induces identifications:
\begin{align*}
\sfV(\widetilde{\Gamma}) &= \FF \times \sfV(\Gamma) \mbox{ by } g\tilde v_{i} \leftrightarrow (g,v_{i}) \mbox{ and} \\
\sfE(\widetilde{\Gamma}) &= \FF \times \sfE(\Gamma) \mbox{ by } g\tilde e_{i} \leftrightarrow (g,e_{i}). 
\end{align*}
%Hence, we have $\FF$--equivariant isomorphisms $\ZZ[V(\widetilde{\Gamma})] = \bigl(\ZZ[\FF]\bigr)^{m}$ and $\ZZ[E(\widetilde{\Gamma})] = \bigl(\ZZ[\FF]\bigr)^{n}$.  
Using these identifications, we can write $\tilde f(\tilde v_{i}) = (g_{i}, v_{j_{i}})$ and the edge-path $\tilde f(\tilde e_{i})$ can be expressed as $\prod_{k} (g_{i,k}, e_{j_{i,k}}^{\varepsilon_{i,k}})$ where $g_{i,k} \in \FF$, $e_{j_{i,k}} \in \sfE(\Gamma)$ and $\varepsilon_{i,k} \in \{-1,1\}$.  We define:
\begin{align*}
\label{eq:derivation}  
\frac{\partial}{\partial v_{j}} f(v_{i}) &= \begin{cases} g_{j} & \mbox{if } j = j_{i} \\ 0 & \mbox{else} \end{cases} \mbox{ and } \\
\frac{\partial}{\partial e_{j}} f(e_{i}) &= \sum_{j = j_{i,k}} \varepsilon_{i,k}g_{i,k}.
\end{align*}

If $\Gamma = \bigvee_{r=1}^{\rank(\FF)} S^{1}$ and we have identified the edges of $\Gamma$ with a basis of $\FF$, then $\frac{\partial}{\partial e_{j}}$ is the usual Fox derivative $\frac{\partial}{\partial x_{j}} \from \FF \to \ZZ[\FF]$.

\begin{definition}
\label{def:J}
The \emph{Jacobian matrices} $J_{0}(f) \in \Mat_{m}(\ZZ[\FF])$ and $J_{1}(f) \in \Mat_{n}(\ZZ[\FF])$ are defined as:
\begin{equation}
 J_{0}(f) = \left[ \frac{\partial}{\partial v_{j}} f(v_{i}) \right]_{1 \leq i, \,j \leq m}  \mbox{ and }
 J_{1}(f) = \left[ \frac{\partial}{\partial e_{j}} f(e_{i}) \right]_{1 \leq i, \,j \leq n}.
\end{equation} 
\end{definition}

The definition of these matrices depends on several choices ($v$, $p$, ordering of vertices and edges) that are suppressed from the notation and do not matter for the sequel. % Of note is that the matrix $J_{1}(f)$ only depends on the homotopy class of $f$ (relative to the other choices).

As $\tilde f$ is a lift of $f$, the following is immediate.  Compare this proposition to \cite[Proposition~9.2]{un:DKL-2}.

\begin{proposition}
\label{prop:transition}
Suppose $f \from \Gamma \to \Gamma$ is a morphism.  Then $J_{1}(f)_{\ell^{1}} = M(f)$. 
\end{proposition}

\begin{example}
\label{ex:theta}
Let $\Gamma$ be the theta graph labeled as in pictured in Figure~\ref{fig:theta}.  A morphism $f \from \Gamma \to \Gamma$ is defined by $f(a) = b\inv$, $f(b) = c\inv$ and $f(c) = a\inv$.  

\begin{figure}[ht]
\centering
\begin{tikzpicture}
\draw[very thick] (0,0) circle [x radius=1.8, y radius=1];
%\draw[very thick] (2.6,0) circle [x radius=0.8, y radius=0.8];
\draw[very thick] (-1.8,0) -- (1.8,0);
\fill (-1.8,0) circle [radius=0.075];
\fill (1.8,0) circle [radius=0.075];
%\dag\draw[thick] (3.3,0.1) -- (3.4,0) -- (3.5,0.1);
\draw[thick] (0.1,1.1) -- (0,1) -- (0.1,0.9);
\draw[thick] (0.1,-0.1) -- (0,0) -- (0.1,0.1);
\draw[thick] (0.1,-1.1) -- (0,-1) -- (0.1,-0.9);
\node at (0.2,1.3) {$a$};
\node at (0.2,0.3) {$b$};
\node at (0.2,-0.7) {$c$};
\node at (2.05,0) {$v$};
\node at (-2.05,0) {$w$};
\end{tikzpicture}
\caption{The graph $\Gamma$ in Example~\ref{ex:DKL}.}
\label{fig:theta}
\end{figure}

We order the vertices and edges alphabetically and let $\tau$ be the edge $b$.  This induces the isomorphism $\pi_{1}(\Gamma,v) \cong \FF = \I{x_{1},x_{2}}$ where $x_{1} = ba\inv$ and $x_{2} = bc\inv$.  Letting $p$ be the edge-path $b$ from $v$ to $f(v) = w$ we find:
\begin{equation*}
J_{0}(f) = \begin{bmatrix}
0 & 1 \\ x_{2} & 0
\end{bmatrix}
\mbox{ and }
J_{1}(f) = \begin{bmatrix}
0 & -1 & 0 \\
0 & 0 & -x_{2} \\
-x_{1} & 0 & 0
\end{bmatrix}.
\end{equation*}
\end{example}

\begin{example}
\label{ex:DKL}
We will construct the Jacobians $J_{0}(f)$ and $J_{1}(f)$ for the ``Running Example'' $f \from \Gamma \to \Gamma$ of \cite{ar:DKL15} (Example~2.2) and \cite{un:DKL-2} (Example~3.3).  The graph $\Gamma$ is shown in Figure~\ref{fig:DKL-graph} (we have reversed the orientation on some edges).  We order the vertices and edges alphabetically.

\begin{figure}[ht]
\centering
\begin{tikzpicture}
\draw[very thick] (0,0) circle [x radius=1.8, y radius=1];
\draw[very thick] (2.6,0) circle [x radius=0.8, y radius=0.8];
\draw[very thick] (-1.8,0) -- (1.8,0);
\fill (-1.8,0) circle [radius=0.075];
\fill (1.8,0) circle [radius=0.075];
\draw[thick] (3.3,0.1) -- (3.4,0) -- (3.5,0.1);
\draw[thick] (0.1,1.1) -- (0,1) -- (0.1,0.9);
\draw[thick] (0.1,-0.1) -- (0,0) -- (0.1,0.1);
\draw[thick] (0.1,-1.1) -- (0,-1) -- (0.1,-0.9);
\node at (0.2,1.3) {$a$};
\node at (0.2,0.3) {$b$};
\node at (3.7,0.1) {$c$};
\node at (0.2,-0.7) {$d$};
\node at (2.05,0) {$v$};
\node at (-2.05,0) {$w$};
\end{tikzpicture}
\caption{The graph $\Gamma$ in Example~\ref{ex:DKL}.}
\label{fig:DKL-graph}
\end{figure}

The morphism $f$ is described by:
\begin{equation*}
a  \mapsto d \qquad
b  \mapsto a \qquad
c  \mapsto ba\inv \qquad
d  \mapsto c\inv ab\inv da\inv b
\end{equation*}
Letting $\tau \subset \Gamma$ be the edge $a$ we have the induced isomorphism $\pi_{1}(\Gamma,v) \cong \FF = \I{x_{1},x_{2},x_{3}}$ where $x_{1} = ab\inv$, $x_{2} = da\inv$ and $x_{3} = c\inv$.  The corresponding automorphism $\Phi \in \Aut(\FF)$ is: 
\begin{equation*}
x_{1} \mapsto x_{2} \qquad
x_{2} \mapsto x_{3}x_{1}x_{2}x_{1}\inv x_{2}\inv \qquad
x_{3} \mapsto x_{1}
\end{equation*}
Fix a lift of $\tilde v \in \widetilde{\Gamma}$ and let $\tilde f$ be the lift of $f$ that fixes $\tilde v$.  Thus $\tilde f \from \widetilde{\Gamma} \to \widetilde{\Gamma}$ is $\Phi$--equivariant in the sense that $\tilde f(x z) = \Phi(x) \tilde f(z)$ for $x \in \FF$ and $z \in \widetilde{\Gamma}$.  

We have $\tilde f(\tilde v) = \tilde v$ and $\tilde f(\tilde w) = x_{2} \tilde w$ and therefore $J_{0}(f) = \left[ \begin{smallmatrix} 1 & 0 \\ 0 & x_{2} \end{smallmatrix}\right]$.  

Clearly, $\tilde f(\tilde a) = \tilde d$, $\tilde f(\tilde b) = \tilde a$.  The edge-paths $\tilde f(\tilde c)$ and $\tilde f(\tilde d)$ are shown in Figure~\ref{fig:DKL-edgepaths}.  With this we find:
\begin{equation*}
J_{1}(f) = \begin{bmatrix}
0 & 0 & 0 & 1 \\
1 & 0 & 0 & 0 \\
-x_{1}\inv & 1 & 0 & 0 \\
x_{3} - x_{3} x_{1}x_{2} & -x_{3}x_{1} + x_{3}x_{1}x_{2} & -x_{3} & x_{3}x_{1}
\end{bmatrix}
\end{equation*}
Compare with the matrix $A(t)$ in \cite[Example~9.5]{un:DKL-2}.

\begin{figure}[ht]
\centering
\begin{tikzpicture}[scale=0.9]
\draw[very thick] (0,2.5) -- (4,2.5) (0,0) -- (12,0);
\foreach \p in {0,1,2}
	\fill (\p*2,2.5) circle [radius=0.075];
\draw[thick] (0.9,2.6) -- (1,2.5) -- (0.9,2.4);
\draw[thick] (3.1,2.6) -- (3,2.5) -- (3.1,2.4);
\foreach \p in {0,1,2,3,4,5,6}
	\fill (\p*2,0) circle [radius=0.075];
\draw[thick] (1.1,0.1) -- (1,0) -- (1.1,-0.1);
\draw[thick] (2.9,0.1) -- (3,0) -- (2.9,-0.1);
\draw[thick] (5.1,0.1) -- (5,0) -- (5.1,-0.1);
\draw[thick] (6.9,0.1) -- (7,0) -- (6.9,-0.1);
\draw[thick] (9.1,0.1) -- (9,0) -- (9.1,-0.1);
\draw[thick] (10.9,0.1) -- (11,0) -- (10.9,-0.1);
\node at (2,3.5) {$\tilde f(\tilde c)$};
\node at (2,1) {$\tilde f(\tilde d)$};
\node at (0,2.2) {$\tilde v$};
\node at (2,2.2) {$x_{1}\inv \tilde w$};
\node at (4,2.2) {$x_{1}\inv \tilde v$};
\node at (1,2.9) {$\tilde b$};
\node at (3,2.9) {$x_{1}\inv \tilde a$};
\node at (0,-0.3) {$\tilde v$};
\node at (2,-0.3) {$x_{3} \tilde v$};
\node at (4,-0.3) {$x_{3} \tilde w$};
\node at (6,-0.3) {$x_{3} x_{1} \tilde v$};
\node at (8,-0.3) {$x_{3} x_{1} x_{2} \tilde w$};
\node at (10,-0.3) {$x_{3}x_{1} x_{2} \tilde v$};
\node at (12.4,-0.3) {$x_{3}x_{1} x_{2} x_{1}\inv \tilde w$};
\node at (1,0.4) {$x_{3} \tilde c$};
\node at (3,0.4) {$x_{3} \tilde a$};
\node at (5,0.4) {$x_{3} x_{1}\tilde b$};
\node at (7,0.4) {$x_{3} x_{1} \tilde d$};
\node at (9,0.4) {$x_{3}x_{1} x_{2} \tilde a$};
\node at (11,0.4) {$x_{3} x_{1} x_{2} \tilde b$};
\end{tikzpicture}
\caption{The edge-paths $\tilde f(\tilde c)$ and $\tilde f(\tilde d)$ from Example~\ref{ex:DKL}.}\label{fig:DKL-edgepaths}
\end{figure}
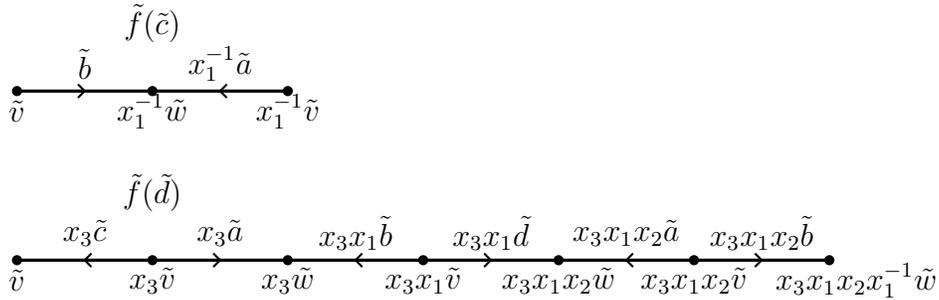
\end{example}

The Jacobian satisfies a chain rule like the usual Fox derivatives~\cite{ar:Fox53}.

\begin{proposition}
\label{prop:chain rule}
Let $f_{1}, f_{2} \from \Gamma \to \Gamma$ be morphisms that both fix the vertex $v$ and suppose that $f_{1}$ represents $\Phi \in \End(\FF)$ where $\pi_{1}(\Gamma,v) \cong \FF$ using the trivial edge-path.  Then \[J_{1}(f_{1} \circ f_{2}) = \Phi\bigl(J_{1}(f_{2})\bigr)J_{1}(f_{1}).\]  
\end{proposition}

\begin{proof}
Use the notation as in the above discussion of $J_{1}(f)$.  Given an edge $\tilde e \in E(\widetilde{\Gamma})$ we write $\tilde f_{2}(\tilde e) = \prod (g_{k},e^{\varepsilon_{k}}_{j_{k}})$.  Since $\tilde f_{1}(gx) = \Phi(g)\tilde f_{1}(x)$ we have that $\tilde f_{1} \circ \tilde f_{2}(\tilde e)$ can be expressed as the concatenation of the edge-paths $\Phi(g_{k})\tilde f_{1}(\tilde e_{j_{k}})$.  
\end{proof}

We record the following elementary consequence of the chain rule.

\begin{corollary}
\label{co:chain rule}
Suppose $f \from \Gamma \to \Gamma$ is a morphism that fixes the vertex $v$ and that represents $\Phi \in \Aut(\FF)$ where $\pi_{1}(\Gamma,v) \cong \FF$ using the trivial edge-path.  Then for all $k \geq 1$ we have $(tJ_{1}(f))^{k} = t^{k}J_{1}(f^{k})$ in $\Mat_{\#\abs{\sfE(\Gamma)}}(\ZZ[\FF \rtimes_{\Phi} \ZZ])$.
\end{corollary}

\begin{proof}
By the chain rule (Proposition~\ref{prop:chain rule}) we have: \[J_{1}(f^{k}) = \Phi\bigl(J_{1}(f^{k-1}) \bigr) J_{1}(f) = t\inv J_{1}(f^{k-1}) t J_{1}(f) \in \Mat_{n}(\ZZ[\FF \rtimes_{\Phi} \ZZ]).\]  By induction, we assume $(tJ_{1}(f))^{k-1} = t^{k-1}J_{1}(f^{k-1})$.  Then we compute:
\begin{align*}
 (tJ_{1}(f))^{k} &= (tJ_{1}(f))^{k-1} \cdot tJ_{1}(f)  = t^{k-1}J_{1}(f^{k-1}) \cdot tJ_{1}
 (f) \\
 & = t^{k} \cdot t\inv J_{1}(f^{k-1}) tJ_{1}(f) = t^{k}J_{1}(f^{k}). \qedhere
\end{align*}
\end{proof}

%\begin{corollary}
%\label{co:inverse}
%Suppose $f \from \Gamma \to \Gamma$ is a morphism that fixes the vertex $v$ and that represents $\Phi \in \Aut(\FF)$ where $\pi_{1}(\Gamma,v) \cong \FF$ using the trivial edge-path and let $n = \#\abs{\sfE(\Gamma)}$.  Then $J_{1}(f) \in \Mat_{n}(\ZZ[\FF])$ is invertible.  Specifically, $J_{1}(f)^{\inv} =  \Phi(J_{1}(f'))$ where $f' \from \Gamma \to \Gamma$ is a morphism that is a homotopy inverse to $f$ rel $v$.
%\end{corollary}
%
%\begin{proof}
%Let $I$ denoted the identity matrix in $\Mat_{n}(\ZZ[\FF])$.  By the chain rule (Proposition~\ref{prop:chain rule}) we have:
%\begin{equation*}
% I = J_{1}(f \circ f') = \Phi\bigl(J_{1}(f')\bigr)J_{1}(f). \qedhere
%\end{equation*}
%\end{proof}

%%%%%%%%%%%%%%%%%%%%%%%%%%%%%%

\subsection{Splitting along strata}\label{sec:split}

We now can state and prove the main result of this section, that the $\ell^{2}$--torsion can be computed using data encoded in the strata $H_{s}$.

\begin{convention}
\label{con:tt}
For this section, we let $f \from \Gamma \to \Gamma$ be a relative train-track map representing $\phi \in \Out(\FF)$ with filtration $\emptyset = \Gamma_{0} \subset \Gamma_{1} \subset \cdots \subset \Gamma_{S} = \Gamma$ and strata $H_{s}$ which is the closure of $\Gamma_{s} - \Gamma_{s-1}$.  We let $m = \#\abs{\sfV(\Gamma)}$, $n = \#\abs{\sfE(\Gamma)}$ and $n_{s} = \#\abs{\sfE(H_{s})}$.  We assume that the image of every vertex of $\Gamma$ is fixed by $f$.  We fix an ordering $\sfV(\Gamma) = \{v_{1}, \ldots v_{m}\}$ such that if $v_{i}$ is fixed and $v_{j}$ is not, then $i < j$, i.e., list the fixed vertices first.  Lastly, we fix an ordering $\sfE(\Gamma) = \{e_{1},\ldots,e_{n}\}$ such that if $e_{i} \in H_{i'}$, $e_{j} \in H_{j'}$ and $i'<j'$ then $i<j$, i.e., list the edges in lower stratum first.  Under the identification $\pi_{1}(\Gamma,v) \cong \FF$ we assume that $v$ is fixed by $f$ and so by $\Phi$ we will denote the automorphism of $\FF$ induced by $f$ using the trivial edge-path.
\end{convention}

As in Section~\ref{sec:rtt}, for each $1 \leq s \leq S$, we let $i_{s}$ denote the smallest index such that $e_{i_{s}} \in \sfE(H_{s})$ and $J_{1}(f)_{s}$ the $n_{s} \times n_{s}$--submatrix of $J_{1}(f)$ with $i_{s} \leq i,j \leq i_{s} - 1 + n_{s}$.  Then $J_{1}(f)$ is lower block triangular with blocks $J_{1}(f)_{s}$ along the diagonal and by Proposition~\ref{prop:transition},  $\bigl(J_{1}(f)_{s}\bigr)_{\ell^{1}} = M(f)_{s}$.  

\begin{remark}
\label{rem:chain rule}
By comparing submatrices, Corollary~\ref{co:chain rule} also shows that $(tJ_{1}(f)_{s})^{k} = t^{k}J_{1}(f^{k})_{s} \in \Mat_{n_{s}}(\ZZ[G_{\phi}])$ for all $k \geq 1$ and $1 \leq s \leq S$.
\end{remark}

By $I$ we denote the identity matrix over a ring with identity (the size and ring will usually be clear from the context).  Abusing notation, we will also denote by $I$ the identity operator on a Hilbert space (the space will usually be clear from the context).  Let $K_{f}$ be the operator on $\bigl(\ell^{2}(G_{\phi})\bigr)^{m}$ induced by right-multiplication by $tJ_{0}(f)$.  Likewise, let $L_{f}$ be the 
operator on $\bigl(\ell^{2}(G_{\phi})\bigr)^{n}$ induced by right-multiplication by $tJ_{1}(f)$ and $L_{f, \, s}$ the the operator on $\bigl(\ell^{2}(G_{\phi})\bigr)^{n_{s}}$ induced by right-multiplication by $tJ_{1}(f)_{s}$ 
  
\begin{lemma}
\label{lem:J0}
Suppose $f \from \Gamma \to \Gamma$ is as in Convention~\ref{con:tt}.  Then:
\begin{enumerate}
\item $I - K_{f} \from \bigl(\ell^{2}(G_{\phi})\bigr)^{m}  \to \bigl(\ell^{2}(G_{\phi})\bigr)^{m}$ is injective; and
\item $\log \det (I - K_{f}) = 0$.
\end{enumerate}
\end{lemma}

\begin{proof}
Let $m'$ be the number of vertices of $\Gamma$ that are fixed by $f$ and let $I - tJ_{0}(f)^{[m']}$ be the $m' \times m'$ submatrix of $I - tJ_{0}(f)$ where $1\leq i,j \leq m'$.  This submatrix is diagonal with entries $1-tx_{i}$ for some $x_{i} \in \FF$.  Then in block form we can write:
\begin{equation*}
I-tJ_{0}(f) = \begin{bmatrix}
I - tJ_{0}(f)^{[m']} & 0 \\ A & I
\end{bmatrix}.
\end{equation*}
where $A$ has size $(m-m') \times m'$.  Let $K^{[m']}_{f}$ be the restriction of $K_{f}$ to $\bigl( \ell^{2}(G_{\phi} \bigr)^{m'}$.    

\noindent (1) \ Let $\bxi = (\xi_{1},\ldots,\xi_{m})$ be an $m$--tuple of functions in $\ell^{2}(G_{\phi})$ and suppose that $K_{f}(\bxi) = \bxi$.  Thus $(\xi_{1},\ldots,\xi_{m'})tJ_{0}(f)^{[m']} = (\xi_{1},\ldots,\xi_{m'})$. Hence there are elements $x_{i} \in \FF$ such that $\xi_{i} \cdot tx_{i} = \xi_{i}$ for $1 \leq i \leq m'$.  Therefore for any $h \in G_{\phi}$ and $k \in \NN$, we have $\xi_{i}(h(tx_{i})^{-k}) = \xi_{i}(h)$.  The $t$--exponent of the elements $h(tx_{i})^{-k} \in G_{\phi}$ are all distinct and hence the elements $h(tx\inv)^{-k}$ are also distinct.  As: \[ \sum_{k \in \NN} \bigl(\xi_{i}(h(tx_{i})^{-k})\bigr)^{2} \leq \sum_{g \in G_{\phi}} \bigl(\xi_{i}(g)\bigr)^{2} \leq \| \xi_{i} \|^{2} < \infty \] we must have $\xi_{i}(h) = 0$.  This shows that $I - K_{f}$ is injective.

\noindent (2) \ 
We have that $\det (I - K_{f}) = \det (I - K^{[m']}_{f})$ by \eqref{eq:det-block} as $\det(I) = 1$.  Letting $T_{i} \from \ell^{2}(G_{\phi}) \to \ell^{2}(G_{\phi})$ be the operator induced by right-multiplication by $1-tx_{i}$ and appealing to \eqref{eq:det-block} again, by induction we have:
\[ \det(I - K_{f}^{[m']}) = \prod_{i=1}^{m'} \det(T_{i}). \]  For each $i$, the operator $T_{i}$ is obtained by induction from the operator $T \from \ell^{2}(\ZZ) \to \ell^{2}(\ZZ)$ defined by $T(\xi) = \xi \cdot (1 - z)$, where $\ZZ = \I{z}$, with respect to the inclusion $\iota_{i} \from z \mapsto tx_{i}$ (see \cite[Section~1.1.5]{bk:Luck02}).  

Then $\detG[G_{\phi}](T_{i}) = \detG[G_{\phi}](\iota_{i}^{*}T) = \detG[\ZZ](T) = 1$ by \cite[Theorem~3.14~(6) \& Example~3.22]{bk:Luck02}.  Thus $\det(I - K_{f}) = 1$ as claimed.
\end{proof}

\begin{theorem}[{cf.~\cite[Theorem~7.29]{bk:Luck02}}]
\label{th:det-splitting}
Suppose $f \from \Gamma \to \Gamma$ is as in Convention~\ref{con:tt}.  Then:
\begin{equation}
\label{eq:det-splitting}
-\rho^{(2)}(G_{\phi}) = \sum_{s=1}^{S} \log \det \bigl(I - L_{f, \, s}\bigr)
\end{equation}
\end{theorem}

\begin{proof}
The proof of Theorem~7.29 in \cite{bk:Luck02} establishes that 
\begin{equation}
\label{eq:luck-det}
\rho^{(2)}(G_{\phi}) = -\log \det \bigl(I - L_{f}\bigr) 
\end{equation} 
in the case that $\Gamma = \bigvee_{r=1}^{\rank(\FF)} S^{1}$.  The argument for \eqref{eq:luck-det} essentially goes through in the general case of a graph; we sketch the essential pieces here for completeness.  

We use the mapping torus of $f \from \Gamma \to \Gamma$ as our topological model for $BG_{\phi}$.  Let $C_{*}$ be the cellular chain complex for $EG_{\phi}$ using the natural cell structure on $BG_{\phi}$.  Let \[\bd \from \ZZ [\sfE(\widetilde{\Gamma})] = \bigl( \ZZ[\FF] \bigr)^{n} \to \ZZ[\sfV(\widetilde{\Gamma})] = \bigl(\ZZ[\FF]\bigr)^{m} \] be the boundary map defined on the basis by $\bd(e) = \bd_{1}(e) - \bd_{0}(e)$.  There is a short exact sequence of $\ZZ[G_{\phi}]$--chain complexes as shown below:
\begin{equation}
\label{eq:ses}
\xymatrix@R=6pt@C=10pt{ & B_{*} & C_{*} & D_{*} & \\
& 0\ar[dd] & 0\ar[dd] & 0\ar[dd] & \\
\\
0\ar[r] & 0\ar[r]\ar[ddd] & \bigl(\ZZ[G_{\phi}]\ar[r]\bigr)^{n}\ar[ddd]^{[I - L_{f} \ \bd ]} & \bigl(\ZZ[G_{\phi}]\ar[r]\bigr)^{n}\ar[r]\ar[ddd]^{I - L_{f}} & 0 \\
\\ \\
0\ar[r] & \bigl(\ZZ[G_{\phi}]\ar[r]\bigr)^{m}\ar[r]\ar[ddd]^{I - K_{f}} & \bigl(\ZZ[G_{\phi}]\ar[r]\bigr)^{n} \oplus \bigl(\ZZ[G_{\phi}]\ar[r]\bigr)^{m} \ar[r]\ar[ddd]^{\scriptsize \begin{bmatrix} -\bd \\ I\!-\!K_{f} \end{bmatrix}} & \bigl(\ZZ[G_{\phi}]\ar[r]\bigr)^{n}\ar[r]\ar[ddd] & 0 
\\ \\
\\
0\ar[r] & \bigl(\ZZ[G_{\phi}]\ar[r]\bigr)^{m}\ar[r]\ar[dd] & \bigl(\ZZ[G_{\phi}]\ar[r]\bigr)^{m}\ar[r]\ar[dd] & 0\ar[r]\ar[dd] & 0 \\
\\
& 0 & 0 & 0 &
}
\end{equation} 
The short exact sequence $0 \to B_{*} \to C_{*} \to D_{*} \to 0$ is still exact after tensoring with $\ell^{2}(G_{\phi})$: 
\begin{equation}
\label{eq:l2-ses}
0 \to B_{*}^{(2)} \to C_{*}^{(2)} \to D_{*}^{(2)} \to 0.  
\end{equation}
The weakly exact long $\ell^{2}$--homology sequence (\cite[Theorem~1.21]{bk:Luck02}) associated to \eqref{eq:l2-ses} implies that $H_{2}^{(2)}(D_{*}^{(2)}) = H_{1}^{(2)}(B_{*}^{(2)})$ as $H_{*}^{(2)}(C_{*}^{(2)}) = 0$ (Theorem~\ref{th:vanish}).  By Lemma~\ref{lem:J0}~(1), $H_{1}^{(2)}(B_{*}^{(2)}) = 0$ and hence $H_{2}^{(2)}(D_{*}^{(2)}) = 0$ as well and so $I - L_{f}$ is injective.

The sum formula for $\ell^{2}$--torsion (Theorem~\ref{th:torsion sum}) implies that $\rho^{(2)}(C_{*}^{(2)}) = \rho^{(2)}(B_{*}^{(2)})  + \rho^{(2)}(D_{*}^{(2)})$.
As $\rho^{(2)}(B_{*}^{(2)}) = \log \det (I - K_{f}) = 0$ by Lemma~\ref{lem:J0}~(2), we have: \[\rho^{(2)}(G_{\phi}) = \rho^{(2)}(C_{*}^{(2)}) = \rho^{(2)}(D_{*}^{(2)}) =
-\log \det (I - L_{f}) \] verifying \eqref{eq:luck-det} in general.

Let $n' = n - n_{S}$ and let $I - tJ_{1}(f)^{[n']}$ by the $n' \times n'$ submatrix of $I - tJ_{1}(f)$ where $1 \leq i,j \leq n'$.  Then in block form we can write:
\[ I - tJ_{1}(f) = \begin{bmatrix}
I - tJ_{1}(f)^{[n']} & 0 \\ A & I - tJ_{1}(f)_{S}
\end{bmatrix} \]
where $A$ has size $n_{S} \times n'$.  Let $L_{f}^{[n']}$ be the restriction of $L_{f}$ to $\bigl( \ell^{2}(G_{\phi}) \bigr)^{n'}$.  Injectivity of $I - L_{f}$ implies that both $I - L_{f}^{[n']}$ and $I - L_{f, \, S}$ are injective.  Therefore:
\begin{equation*}
\det (I - L_{f}) = \det(I - L_{f}^{[n']}) \det(I - L_{f, \, S})
\end{equation*}  
using \eqref{eq:det-block}.  By induction, the theorem follows.
\end{proof}

Combining Lemma~\ref{lem:det-bound} and Theorem~\ref{th:det-splitting} we get:

\begin{corollary}
\label{co:compute}
Suppose $f \from \Gamma \to \Gamma$ is as in Convention~\ref{con:tt}.  Then:
\begin{equation}
\label{eq:compute}
-\rho^{(2)}(G_{\phi}) \leq \sum_{s=1}^{S} n_{s}\log \opnorm{I - L_{f, \, s}}.
\end{equation}
\end{corollary}

\begin{remark}
\label{rem:poly}
The vanishing of $\rho^{(2)}(G_{\phi})$ for polynomially growing $\phi \in \Out(\FF)$ can be extracted from the proofs of Lemma~\ref{lem:J0}~(2) and Theorem~\ref{th:det-splitting}.  The key point is that in this case, after passing to some power of $f$, the matrix $I - J_{1}(f)_{s}$ is a $1 \times 1$ matrix whose entry is either 1 or $1 - tx_{s}$ for some $x_{s} \in \FF$ depending on whether or not $M(f)_{s}$ is the zero matrix.  The determinant of such an operator is 1.  Theorem~\ref{th:bound} also shows that the $\ell^{2}$--torsion vanishes in this case as well.
\end{remark}

\section{Upper bound on $-\rho^{(2)}(G_{\phi})$}
\label{sec:bounds}

We can now prove the main result of this article.

\begin{theorem}
\label{th:bound}
Suppose $f \from \Gamma \to \Gamma$ is a relative train-track map with respect to the filtration $\emptyset = \Gamma_{0} \subset \Gamma_{1} \subset \cdots \subset \Gamma_{S} = \Gamma$ representing the outer automorphism $\phi \in \Out(\FF)$.  Let $n_{s}$ denote the number of edges of $\Gamma_{s} - \Gamma_{s-1}$, $\lambda(f)_{s}$ the Perron--Frobenius eigenvalue of $M(f)_{s}$ and $\PF(f)$ the set of indices of the exponentially growing strata. Then:
\begin{equation}
\label{eq:bound}
-\rho^{(2)}(G_{\phi}) \leq \sum_{s \in \PF(f)} n_{s} \log \lambda(f)_{s}.
\end{equation} 
In particular, if $f \from \Gamma \to \Gamma$ does not have an exponentially growing stratum then $\rho^{(2)}(G_{\phi}) = 0$.
\end{theorem}

\begin{proof} 
Let $f \from \Gamma \to \Gamma$ be a relative train-track map representing $\phi$.  By Theorem~\ref{th:torsion} for all $k \in \NN$, we have $\rho^{(2)}(G_{\phi^{k}}) = k\rho^{(2)}(G_{\phi})$.  By Lemma~\ref{lem:rtt}, for all $1 \leq s \leq S$, and $k \in \NN$, we have $\log \lambda(f^{k})_{s} = \log \bigl(\lambda(f)_{s}\bigr)^{k} = k \log \lambda(f)_{s}$.  Thus we are free to replace $f$ by a power if necessary.  Hence we replace $f$ by a power to assume that all periodic vertices of $\Gamma$ are fixed.  Therefore, we can assume that $f$ satisfies Convention~\ref{con:tt} and we may apply Corollary~\ref{co:compute}.   

For $k \in \NN$, let $L_{f^{k},s} \from \bigl( \ell^{2}(G_{\phi})\bigr)^{n_{s}} \to \bigl( \ell^{2}(G_{\phi})\bigr)^{n_{s}}$ be the operator induced by right-multiplication by $t^{k}J_{1}(f^{k})_{s}$.  By Corollary~\ref{co:chain rule}, we have $t^{k}J_{1}(f^{k})_{s} = \bigl(tJ_{1}(f)_{s}\bigr)^{k}$ for all $k \in \NN$ and hence $L_{f^{k}, \, s} = \bigl(L_{f,\, s}\bigr)^{k}$.    

By Proposition~\ref{prop:transition}, we have $\bigl(J_{1}(f^{k})_{s}\bigr)_{\ell^{1}} = M(f^{k})_{s}$.  Hence, by Lemma~\ref{lem:rtt} we have $\bigl(t^{k}J_{1}(f^{k})_{s}\bigr)_{\ell^{1}} = M(f^{k})_{s} =\bigl(M(f)_{s}\bigr)^{k}$ for all $k \in \NN$.  Since the entries in $J_{1}(f^{k})$ are in $\ZZ[\FF]$, the value at the identity in $G_{\phi}$ of any entry in $t^{k}J_{1}(f^{k})$ is zero.  Thus we find that:
\begin{equation}
\label{eq:no cancel}
\bigl(I - t^{k}J_{1}(f^{k})_{s}\bigr)_{\ell^{1}} = I + \bigl(t^{k}J_{1}(f^{k})_{s}\bigr)_{\ell^{1}} = I + \bigl(M(f)_{s}\bigr)^{k}.
\end{equation}

Now using Theorem~\ref{th:torsion}, Corollary~\ref{co:compute}, Proposition~\ref{prop:operator bound} and \eqref{eq:no cancel} we find that for all $k \in \NN$:
\begin{align*}
-\rho^{(2)}(G_{\phi}) & = -\frac{1}{k} \rho^{(2)}(G_{\phi^{k}}) \\
& \leq \frac{1}{k}\sum_{s=1}^{S} n_{s}\log \opnorm{I - L_{f^{k}, \, s}}\\
& \leq \sum_{s=1}^{S} n_{s}\log \opnorm{\bigl(I - t^{k}J_{1}(f^{k})_{s}\bigr)_{\ell^{1}}}^{1/k} \\
& = \sum_{s=1}^{S} n_{s} \log \opnorm{I + \bigl(M(f)_{s}\bigr)^{k}}^{1/k}.
\end{align*}
For $s \in \PF(f)$, we have $r(M(f)_{s}) = \lambda(f)_{s} > 1$, and so by Lemma~\ref{lem:I + A} it follows that $\| I + \bigl(M(f)_{s}\bigr)^{k} \|^{1/k} \to \lambda(f)_{s}$ as $k \to \infty$.  For $s \notin \PF(f)$, $M(f)_{s}$ has finite order.  Therefore $\| I + \bigl(M(f)_{s}\bigr)^{k} \|$ is bounded and hence $\| I + \bigl(M(f)_{s}\bigr)^{k} \|^{1/k} \to 1$ as $k \to \infty$.  

This proves \eqref{eq:bound}, and hence the theorem. 
\end{proof}

%Theorems~\ref{th:iwip} and \ref{th:non-eg} from the Introduxtion follow imediately.

%%%%%%%%%%%%%%%%%%%%%%%%%%%%%%%%%%%%%%%%%%%%%%%%%%%%%%%%%%%%%%%%%%%%%%%%%%%%%

\bibliography{torsion-upper}
\bibliographystyle{acm}

\end{document}